\documentclass[11pt]{amsart}

\usepackage{graphicx}
\usepackage{pinlabel}

\newtheorem{theorem}{Theorem}[section]
\newtheorem{proposition}[theorem]{Proposition}
\newtheorem{lemma}[theorem]{Lemma}
\newtheorem{corollary}[theorem]{Corollary}

\theoremstyle{definition}
\newtheorem{definition}[theorem]{Definition}

\newcommand{\A}{\mathcal{A}}

\newcommand{\dist}{\operatorname{dist}}

\begin{document}

\title[Arc distance equals level number]
{Arc distance equals level number}

\author{Sangbum Cho}
\address{Department of Mathematics\\
University of California at Riverside\\
Riverside, California 92521\\
USA}
\email{scho@math.ucr.edu}

\author{Darryl McCullough}
\address{Department of Mathematics\\
University of Oklahoma\\
Norman, Oklahoma 73019\\
USA}
\email{dmccullough@math.ou.edu}
\urladdr{www.math.ou.edu/$_{\widetilde{\phantom{n}}}$dmccullough/}
\thanks{The second author was supported in part by NSF grant DMS-0802424}

%    Information for third author
\author{Arim Seo}
\address{Department of Mathematics\\
California State University at San Bernar\-di\-no\\
San Bernar\-di\-no, California 92407\\
USA}
\email{aseo@csusb.edu}


\date{\today}


\begin{abstract}
Let $K$ be a knot in $1$-bridge position with respect to a genus-$g$
Heegaard surface that splits a $3$-manifold $M$ into two handlebodies $V$
and $W$. One can move $K$ by isotopy keeping $K\cap V$ in $V$ and $K\cap
W$ in $W$ so that $K$ lies in a union of $n$ parallel genus-$g$ surfaces
tubed together by $n-1$ straight tubes, and $K$ intersects each tube in two
arcs connecting the ends. We prove that the minimum $n$ for which this is
possible is equal to a Hempel-type distance invariant defined using the arc
complex of the two holed genus-$g$ surface.
\end{abstract}

\maketitle

\section*{Introduction}
\label{sec:intro}

A knot $K$ in a closed orientable $3$-manifold $M$ is said to be in
\textit{$1$-bridge position} with respect to a surface $F$ if $F$ is a
Heegaard surface that splits $M$ into two handlebodies $V$ and $W$, and
each of $K\cap V$ and $K\cap W$ is a single arc that is parallel into
$F$. We denote the $1$-bridge position of $K$ with respect to $F$ by $(F,
K)$, and the \textit{genus} of $(F,K)$ is the genus of $F$. A knot is
called a \textit{$(g, 1)$-knot} if it can be put in genus-$g$ $1$-bridge
position.

There is a natural way to reposition a knot in $1$-bridge position, called
\textit{level position.}  In a neighborhood $F\times [0,1]$ of $F$ in
$M$, one may take $n$ parallel copies of the form $F\times \{t\}$ and
tube them together with $n-1$ unknotted tubes to obtain a surface $G$ of
genus $gn$ in $F\times[0,1]$, where $g$ is the genus of $F$. We say that
$K$ lies in \textit{$n$-level position} with respect to $F$ if $K\subset
G$, and moreover $K$ meets each of the $n-1$ tubes in two arcs, each of
which connects the two ends of the tube. As we will see below, every
$1$-bridge position of $K$ is isotopic keeping $K\cap V$ in $V$ and $K\cap
W$ in $W$ into some $n$-level position. The minimum such $n$ is an
invariant of the $1$-bridge position, called the \textit{level number.} Of
course, the minimum level number over all genus-$g$ $1$-bridge positions of
a $(g, 1)$-knot is an invariant of the knot.

Level position was used by
M.~Eudave-Mu\~noz~\cite{Eudave-Munoz1,Eudave-Munoz2} to obtain closed
incompressible surfaces in the complements of $(1,1)$-knots.

In this note, we use an invariant of a $1$-bridge position, called its
\textit{arc distance.} This is a version of a well-known complexity of a
Heegaard splitting introduced by J. Hempel in \cite{Hem} and defined using
the curve complex of the Heegaard surface. D. Bachman and S. Schleimer have
used a more general and somewhat different definition of arc distance to
obtain information about bridge positions of knots~\cite{B-S}. To define
our arc distance, write $K\cap F=\{x,y\}$. The isotopy classes of arcs in
$F$ from $x$ to $y$ form the vertices of a simplicial complex called the
\textit{arc complex} of $F - \{x, y\}$. The arc distance
of the $1$-bridge position is the minimum distance (simplicial distance
in the $1$-skeleton of the arc complex) between the collection of vertices
represented by arcs in $F$ from $x$ to $y$ that are parallel to $K \cap V$
in $V$ and the analogous collection for $K \cap W$.

Our main result, Theorem~\ref{thm:nLevels}, says that the arc distance of a
$1$-bridge position of $K$ equals its level number. Although the proof is
not especially difficult, this fact seems noteworthy in that although many
such Hempel-type invariants have been defined and used, this appears to be
the first that gives a concrete and natural geometric meaning to every
possible value of the invariant, rather than just small values.

Theorem~\ref{thm:nLevels} for the case $g=1$ appeared in the third author's
dissertation.

We are grateful to the referee for a careful reading and for suggesting
improvements to the manuscript.

\section{Leveling a $(g, 1)$-knot}
\label{sec:leveling}

Suppose that $K$ is in $1$-bridge position with respect to $F$, which
splits $M$ into two handlebodies $V$ and $W$. A \textit{shadow} of $K\cap
V$ is an arc in $F$ isotopic to $K\cap V$, relative to $K\cap F$, through
arcs in $V$. A shadow of $K \cap W$ is defined similarly. A
\textit{Heegaard isotopy} of $K$ is a (piecewise-linear) isotopy of $K$
such that $K\cap V$ stays in $V$ and $K\cap W$ stays in $W$ at all times.
The resulting knot may not be in strict $1$-bridge position, since the arc
$K\cap V$ may be moved to meet $F$ in its interior or even to be a shadow
of $K \cap V$.

A \textit{$1$-leveling} of a knot $K$ in $1$-bridge position with respect
to $F$ is a Heegaard isotopy that ends with a knot $K'\subset F$. For
$n\geq 2$, an \textit{$n$-leveling} of $K$ is a Heegaard isotopy taking $K$
to a knot $K'$ which may be described as follows: Fix a collar $F\times
[0,1]$ of $F$ in $W$, with $F=F\times\{0\}$. Let $0=t_1<t_2<\cdots
<t_n=1$ be a sequence of values, and put $F_i=F\times \{t_i\}\subset
F\times [0,1]$. Let $D_1,\ldots\,$, $D_{n-1}$ be a collection of disks in
$F$ with $D_i\cap D_{i+1}=\emptyset$.  Denote by $T_j$ the tube $\partial
D_j\times [t_j,t_{j+1}]$ connecting $F_j$ and $F_{j+1}$ for each $1\leq
j\leq n-1$. From the union $F_1 \cup T_1 \cup \cdots \cup F_{n-1} \cup
T_{n-1} \cup F_n$, remove the interiors of $D_j \times \{t_j\}$ and $D_j
\times \{t_{j+1}\}$ for $1\leq j\leq n-1$ to get a closed surface $G$ of
genus $gn$ where $g$ is the genus of $F$. Then
\begin{enumerate}
\item $K'\subset G$
\item $K'\cap T_j$ consists of two arcs, each connecting two boundary
circles of $T_j$, for each $1\leq j\leq n-1$.
\end{enumerate}
\noindent Necessarily, $K'\cap F_1$ and $K'\cap F_n$ are single arcs, and
$K'\cap F_i$ is a pair of arcs for each $2\leq i\leq n-1$. The knot $K'$ is
said to be in \textit{$n$-level position} with respect to~$F$.

If $K$ is in level position with respect to $F$, then there is a knot in
$1$-bridge position with respect to $F$ which is Heegaard isotopic to $K$.
Conversely, we have
\begin{proposition}
Let $K$ be in $1$-bridge position with respect to $F$. Let $n$ be the
minimum number of intersection points of shadows $\alpha_V$ and $\alpha_W$
of $K\cap V$ and $K\cap W$ respectively. Then $K$ is Heegaard isotopic to a
knot in $k$-level position with respect to $F$ for some $k < n$.\par
\label{prop:level}
\end{proposition}
\noindent We will not give a direct proof of Proposition~\ref{prop:level}.
Although such a proof is not difficult, it is somewhat cumbersome to
explain and tedious to read. And it is not needed, for as we will see,
Proposition~\ref{prop:level} follows directly from our main result,
Theorem~\ref{thm:nLevels}, together with the connectivity of the arc
complex discussed in Section~\ref{sec:arc_complex} below.

In view of Proposition~\ref{prop:level}, we may make the following
definition for a knot $K$ in genus-$g$ $1$-bridge position with respect to
$F$:
\begin{enumerate}
\item The \textit{level number} of the $1$-bridge position $(F, K)$ is the
minimum $n$ such that $K$ is Heegaard isotopic to a knot in $n$-level
position with respect to $F$.
\item The \textit{genus-$g$ level number} of $K$ is the minimum level
number over all genus-$g$ $1$-bridge positions of $K$.
\end{enumerate}

\section{The arc complex}
\label{sec:arc_complex}

Let $\Sigma$ be a genus-$g$ surface with two holes, $g \geq 0$, and denote
by $C_1$ and $C_2$ the two boundary circles of $\Sigma$. The \textit{arc
complex} $\A(\Sigma)$ of $\Sigma$ is a simplicial complex defined as
follows.  The vertices are isotopy classes of properly embedded arcs in
$\Sigma$ connecting $C_1$ and $C_2$, and a collection of $k+1$ vertices
spans a $k$-simplex if it admits a collection of representative arcs which
are pairwise disjoint.  In this section we will show that $\A(\Sigma)$ is
connected. Indeed, as we will explain, it is contractible.

Arc complexes have been used in Teichm\"uller theory by J. Harer
\cite{Harer2,Harer1} (see also A. Hatcher \cite{Hatcher}) and R. C. Penner
\cite{Penner1}. In particular, many arc complexes are known to be
contractible, although we have not found our particular case in the
existing literature.

Let $v$ and $w$ be vertices of $\mathcal A(\Sigma)$.  Define $v \cdot w$ to
be the minimal cardinality of $l \cap m$ where $l$ and $m$ are arcs in
$\Sigma$ which represent $v$ and $w$, respectively, and intersect
transversely.
\begin{figure}
\labellist
\pinlabel \large $C_1$ [B] at 49 119
\pinlabel \large $C_1$ [B] at 205 119
\pinlabel \large $C_2$ [B] at 40 -9
\pinlabel \large $C_2$ [B] at 196 -9
\pinlabel $l$ [B] at 30 99
\pinlabel $m$ [B] at 72 99
\pinlabel $p$ [B] at 42 70
\pinlabel $m'$ [B] at 179 66
\endlabellist
\begin{center}
\includegraphics{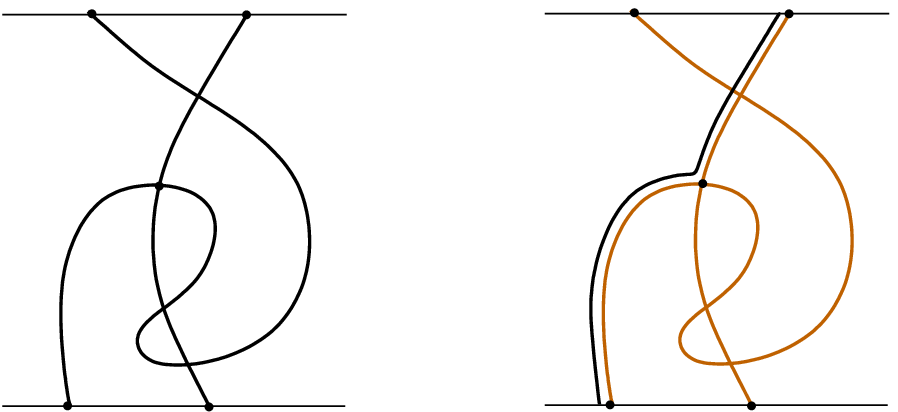}
\caption{} \label{conn}
\end{center}
\end{figure}

\begin{lemma}
Let $v$ and $w$ be vertices of $\mathcal A(\Sigma)$ and suppose $v \cdot w
> 0$. Then there exists a vertex $w'$ such that $w \cdot w' = 0$ and $w'
\cdot v < w \cdot v$.\par
\label{lem:connected}
\end{lemma}

\begin{proof}
Choose arcs $l$ and $m$ representing the vertices $v$ and $w$,
respectively, so that $|l \cap m| = v \cdot w$.  Since $v \cdot w > 0$, we
have at least one intersection point of $l$ and $m$. Let $p$ be the
intersection point for which the sub-arc of $l$ connecting $p$ and $C_2$ is
disjoint from $m$.  Denote by $m'$ the union of this sub-arc and the
sub-arc of $m$ connecting $p$ and $C_1$ (see Figure \ref{conn}).  Then the
arc $m'$ is disjoint from $m$ and has fewer intersections with $l$ than $m$
had (after a slight isotopy) since at least $p$ no longer counts. Letting
$w'$ be the vertex represented by $m'$, we have $w' \cdot v < w \cdot v$
and $w \cdot w' = 0$.
\end{proof}

\begin{theorem}
The arc complex $\mathcal A(\Sigma)$ is connected. In fact, if
representative arcs of $v$ and $w$ intersect transversely in $k$ points,
then the distance from $v$ to $w$ is at most~$k+1$.\par
\label{thm:connected}
\end{theorem}

\begin{proof}
Let $v$ and $w$ be any two vertices of $\mathcal A(\Sigma)$.  If $v \cdot w
= 0$, then $v$ and $w$ are connected by an edge of $\mathcal A(\Sigma)$, so
lie at distance $1$. If $v \cdot w = k > 0$, then Lemma
\ref{lem:connected} and induction give the result.
\end{proof}

In fact, $\mathcal A(\Sigma)$ is contractible. This can be proven fairly
quickly using Proposition~3.1 of~\cite{C}. Since we do not need this fact,
we do not include the argument.

\section{The arc distance of a $(g, 1)$-knot}
\label{sec:distance}

In Section \ref{sec:arc_complex}, we showed that the arc complex
$\A(\Sigma)$ is connected. Thus, for any two vertices $v$ and $w$ of
$\A(\Sigma)$, we can define the \textit{distance} between $v$ and $w$,
$\dist (v, w)$, to be the distance in the 1-skeleton of $\A(\Sigma)$ from
$v$ to $w$ with the usual path metric.

Keeping the notation of previous sections, let $K$ be a $(g, 1)$-knot in
$1$-bridge position with respect to the Heegaard surface $F$. By removing
from $F$ a small open neighborhood of the two points $K \cap F$, we obtain
a 2-holed genus-$g$ surface $\Sigma$. Denote by $k$ and $k'$ the two arcs
$V \cap K$ and $W \cap K$, and let $s$ and $s'$ be shadows of $k$ and $k'$
respectively.  Then the arcs $s \cap \Sigma$ and $s' \cap \Sigma$ represent
vertices of the arc complex $\A(\Sigma)$. We will call $s \cap
\Sigma$ and $s' \cap \Sigma$ \textit{shadows} of $k$ and $k'$ again.

\begin{definition}
Let $K$ be in genus-$g$ $1$-bridge position with respect to $F$.
\begin{enumerate}
\item The \textit{arc distance} of $(F,K)$ is the minimum of $\dist (v,
v')$ over all the vertices $v$ and $v'$ represented by shadows of $K \cap
V$ and $K \cap W$, respectively.
\item The \textit{genus-$g$ arc distance} of $K$ is the minimum of the arc
distance of $(F, K)$ over all genus-$g$ $1$-bridge positions $(F, K)$ of
$K$.
\end{enumerate}
\end{definition}

We observe that the trivial knot is the only knot of arc distance $0$, and
a knot in $S^3$ has genus-$1$ arc distance $1$ if and only if it is
a nontrivial torus knot. Figure \ref{distance_example} shows that the
genus-$1$ arc distance of the figure-$8$ knot is at most $2$, and hence is
$2$ since the figure-$8$ knot is not a torus knot.
\begin{figure}
\labellist
\pinlabel $k$ [B] at 58 170
\pinlabel $k'$ [B] at 104 101
\pinlabel $s$ [B] at 223 154
\pinlabel $s'$ [B] at 272 92
\pinlabel $s''$ [B] at 224 35
\pinlabel $s''$ [B] at 232 9
\pinlabel $s$ [B] at 189 9
\pinlabel $s'$ [B] at 271 9
\endlabellist
\begin{center}
\includegraphics{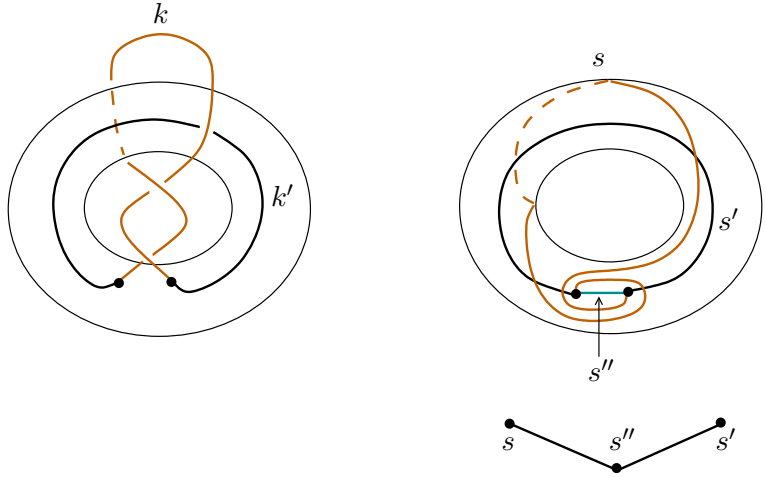}
\caption{A genus-$1$ $2$-level position of the figure-$8$ knot, having
arc distance $2$.}
\label{distance_example}
\end{center}
\end{figure}

\begin{theorem} Let $K$ be a nontrivial knot which is in $1$-bridge
position with respect to $F$.  If $K$ is in $n$-level position with respect
to $F$, then the arc distance of $(F, K)$ is at most $n$. Conversely, if
the arc distance of $(F, K)$ is $n$, then $K$ is Heegaard isotopic to a
knot in $n$-level position with respect to $F$. As a consequence, the arc
distance of $(F, K)$ equals the level number of $(F, K)$.\par
\label{thm:nLevels}
\end{theorem}

\begin{proof}
Suppose that $K$ is in $n$-level position with respect to $F$.
The case of $n = 1$ is clear. We will assume that $n \geq 3$. (The case of
$n = 2$ is similar but simpler.)
We describe
the surface $G$ as in Section \ref{sec:leveling}.  In particular, recall
that the tube $T_j$ connects two surfaces $F_j$ and $F_{j+1}$.  By an
isotopy, we may assume that the two arcs $K \cap T_j$ are vertical, that is
$K \cap T_j = (K \cap \partial D_j)\times [t_j, t_{j+1}]$.  Denote the arcs
$K \cap F_1$ and $K \cap F_n$ by $k$ and $k'$ respectively, and denote
the two arcs of $F_j \cap K$ by $\alpha_j$ and $\beta_j$ for each $2 \leq
j \leq n-1$.  Choose an arc $\mu_j$ properly embedded in $D_j \times
\{t_j\}$, connecting the two points $K \cap (\partial D_j \times \{t_j \})$
for each $1 \leq j \leq n-1$ (see Figure~\ref{dist1}).
\begin{figure}
\labellist
\pinlabel \Large $F_1$ [B] at 9 130
\pinlabel \Large $F_2$ [B] at 9 83
\pinlabel \Large $F_3$ [B] at 9 34
\pinlabel \large $T_1$ [B] at 115 124
\pinlabel \large $T_2$ [B] at 195 76
\pinlabel \large $T_3$ [B] at 277 28
\pinlabel $a$ [B] at 98 142
\pinlabel $b$ [B] at 79 121
\pinlabel $k$ [B] at 46 127
\pinlabel $k$ [B] at 260 143
\pinlabel $\mu_1$ [B] at 84 140
\pinlabel $\mu_2$ [B] at 164 94
\pinlabel $\mu_3$ [B] at 249 47
\pinlabel $\alpha_2$ [B] at 130 94
\pinlabel $\beta_2$ [B] at 57 87
\pinlabel $\beta_2$ [B] at 244 91
\pinlabel $\alpha_3$ [B] at 98 40
\pinlabel $\alpha_3$ [B] at 277 48
\pinlabel $\beta_3$ [B] at 208 45
\endlabellist
\begin{center}
\includegraphics{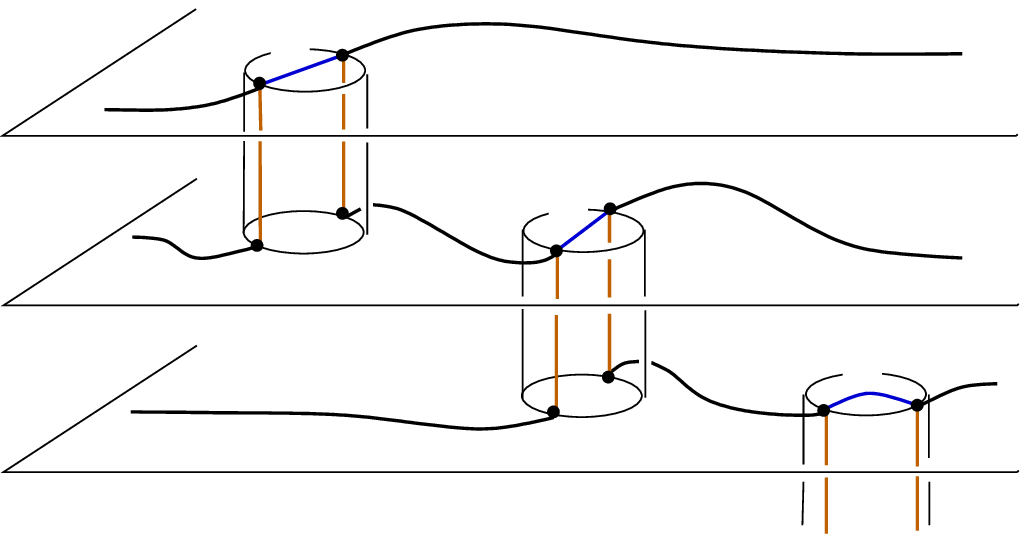}
\caption{}
\label{dist1}
\end{center}
\end{figure}

Let $a = a \times \{t_1\}$ and $b = b \times \{t_1\}$ be the endpoints of
$k$, with notation chosen so that $a\times\{t_2\}\in \alpha_2$ and
$b\times\{t_2\}\in \beta_2$. There is an isotopy $j_t$ of $F_2$ that moves
the endpoints of $\mu_2$ along $\alpha_2$ and $\beta_2$ until they reach $a
\times \{t_2\}$ and $b \times \{t_2\}$, stretching $\mu_2$ onto
$\alpha_2\cup \mu_2\cup \beta_2$. Extend $j_t$ to the isotopy $J_t =
j_t\times id_{[t_2,t_n]}$ on $F \times [t_2, t_n]$.

Consider the knot obtained from $K$ by replacing $K\cap (F\times [t_2,t_n])$
by $J_1(K\cap (F\times [t_2,t_n]))$. The original $K$ is isotopic to this new
knot by an isotopy supported on a small neighborhood of $F\times
[t_2,t_n]$, that resembles $J_t$ on $F\times [t_2,t_n]$. This isotopy pulls
$\alpha_2\cup \beta_2$ onto part of $K\cap T_1$, and stretches $\mu_2$ onto
$\alpha_2\cup \mu_2\cup \beta_2$, as $J_t$ did.

Calling the new knot $K$ again, we may notationally replace each
$\mu_2,\ldots\,$, $\mu_{n-1}$ and $k'$ by its image under $J_1$, each
$D_2,\ldots\,$, $D_{n-1}$ by its image, and so on. The new $\alpha_3$ and
$\beta_3$ end at $a\times\{t_3\}$ and~$b\times \{t_3\}$.

Repeat this process on each descending level. At the last stage (after
renaming), $K$ has been moved to $k \cup (a \cup b) \times [t_1, t_n]
\cup k'$ and we have the sequence of arcs $k$, $\mu_1,\ldots, \mu_{n-1}$,
$k'$, with endpoints lying in $a \times [t_1, t_n]$ and $b \times [t_1,
t_n]$.  After projecting $k$, $\mu_1,\ldots\,$, $\mu_{n-1}$, and $k'$ to
$F$, each intersects the next only in their endpoints. Therefore the
vertices represented by the projected arcs $k$ and $k'$ have distance at
most $n$ in the arc complex.

The projected $k$ and $k'$ are shadows of $K \cap V$ and $K \cap W$, where
$V$ and $W$ are the two handlebodies into which $F$ cuts $M$. Thus the
arc distance of $(F, K)$ is at most $n$.

Conversely, suppose that the arc distance of $(F, K)$ is $n$ for $n \geq 3$
(again the case $n=1$ is clear and we omit the case $n = 2$, which is
similar to $n\geq 3$). Denote by $p$ and $q$ the two points $K \cap
F$. Then we have a sequence of arcs $s_0, s_1, s_2, \cdots, s_{n-1}, s_n$
in $F$, each connecting $p$ and $q$, such that $s_0$ and $s_n$ are shadows
of $V \cap K$ and $W \cap K$, and $s_{j-1}$ meets $s_j$ only in their
endpoints $p$ and $q$ for $1 \leq j \leq n$.

Let $N_p$ and $N_q$ be disjoint regular neighborhoods of $p$ and $q$ in $F$
respectively. By a Heegaard isotopy, we may assume that each of $N_p \cap
(s_0 \cup s_1 \cup \cdots \cup s_n)$ and $N_q \cap (s_0 \cup s_1 \cup
\cdots \cup s_n)$ is contractible. In particular, any $s_i$ and $s_j$ meet
in $N_p$ only at the point $p$, and in $N_q$ only at the point $q$.  For $1
\leq j \leq n-1$, choose regular neighborhoods $D_j$ of $s_j \cap
\overline{F - (N_p \cup N_q)}$ in $\overline{F - (N_p \cup N_q)}$ so that
$s_0$ is disjoint from $D_1$, $s_n$ is disjoint from $D_{n-1}$, and
$D_{j-1}$ is disjoint from $D_j$.  For $1 \leq j \leq n-1$, denote the arcs
$s_j \cap N_p$ and $s_j \cap N_q$ by $\alpha_j$ and $\beta_j$ respectively,
and the points $\alpha_j \cap \partial N_p$ and $\beta_j \cap \partial N_q$
by $p_j$ and $q_j$ respectively (see Figure~\ref{dist2}).
\begin{figure}
\labellist
\pinlabel \Large $D_j$ [B] at 100 121
\pinlabel \Large $D_{j-1}$ [B] at 170 15
\pinlabel \Large $N_p$ [B] at 15 5
\pinlabel \Large $N_q$ [B] at 272 5
\pinlabel $p$ [B] at 43 44
\pinlabel $\alpha_j$ [B] at 46 76
\pinlabel $\alpha_{j-1}$ [B] at 68 38
\pinlabel $p_j$ [B] at 67 82
\pinlabel $p_{j-1}$ [B] at 51 21
\pinlabel $s_j$ [B] at 127 113
\pinlabel $s_{j-1}$ [B] at 117 16
\pinlabel $s_0$ [B] at 285 119
\pinlabel $q$ [B] at 240 44
\pinlabel $\beta_j$ [B] at 238 74
\pinlabel $\beta_{j-1}$ [B] at 222 60
\pinlabel $q_j$ [B] at 217 80
\pinlabel $q_{j-1}$ [B] at 212 48
\endlabellist
\begin{center}
\includegraphics{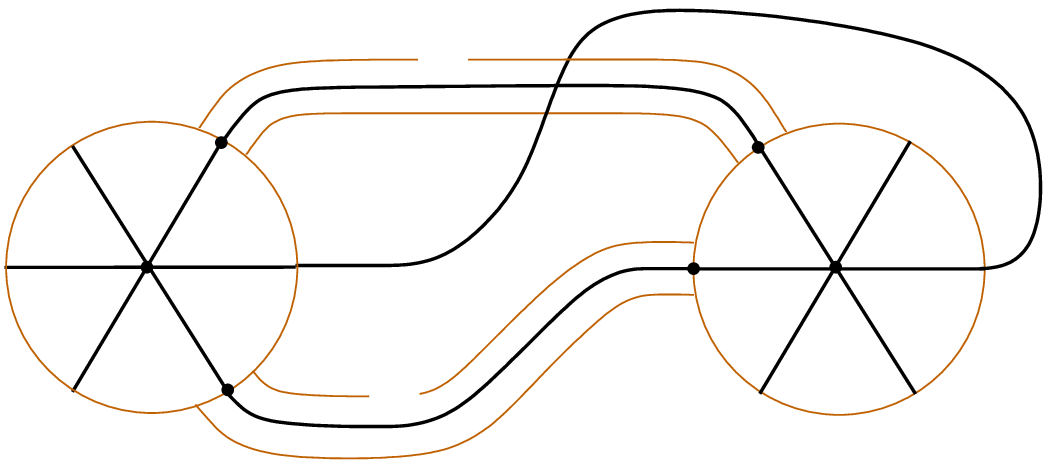}
\caption{}
\label{dist2}
\end{center}
\end{figure}

As in Section \ref{sec:leveling}, let $0 = t_1<t_2<\cdots <t_n=1$ be a
sequence of values, put $F_j = F \times \{t_j\}\subset F\times [0,1]\subset
W$, and construct a closed surface $G$ from the surfaces $F_j$ and the
tubes $T_j = \partial D_j \times [t_j, t_{j+1}]$. By a Heegaard isotopy, we
may assume that $K = s_0 \times \{t_1\} \cup (p \cup q) \times [t_1, t_n]
\cup s_n \times \{t_n\}$. Construct a knot $K'$ contained in $G$ so that:
\begin{enumerate}
\item $K' \cap F_1 = (s_0 \cup \alpha_1\cup \beta_1) \times \{t_1\} $,
\item $K' \cap F_j = (\alpha_{j-1} \cup \alpha_j \cup \beta_{j-1} \cup
\beta_j) \times \{t_j\}$, for $2 \leq j \leq n-1$,
\item $K' \cap F_n = (s_n \cup \alpha_{n-1} \cup \beta_{n-1}) \times
\{t_n\}$, and
\item $K' \cap T_j = (p_j \cup q_j) \times [t_j, t_{j+1}]$, for $1 \leq j
\leq n-1$.
\end{enumerate}
By construction, $K'$ lies in $n$-level position with respect to $F$.
There is a Heegaard isotopy from $K$ to $K'$ that moves each $\{p\}\times
[t_i,t_{i+1}]$ onto $\alpha_i\times \{t_i\}\cup \{p_i\}\times [t_i,t_{i+1}]
\cup \beta_{i+1}\times \{t_{i+1}\}$ and similarly for $\{q\}\times
[t_i,t_{i+1}]$.
\end{proof}

As we mentioned in Section \ref{sec:leveling}, Proposition~\ref{prop:level}
follows from Theorem~\ref{thm:nLevels}. For if $\alpha_V$ and $\alpha_W$
intersect in $n$ points, then as representative arcs of the vertices of the
arc complex $\mathcal{A}(\Sigma)$ they intersect in $n-2$ points. By
Theorem~\ref{thm:connected}, the distance from $\alpha_V$ to $\alpha_W$ is
at most $n-1$, so by Theorem~\ref{thm:nLevels} $K$ is Heegaard isotopic to
a knot in $k$-level position for some $k < n$.

From Theorem~\ref{thm:nLevels}, we have our main objective.
\begin{corollary}
Let $K$ be a nontrivial knot which can be put in genus-$g$ $1$-bridge
position. Then the genus-$g$ arc distance of $K$ equals the genus-$g$ level
number of~$K$.\par
\label{cor:equal}
\end{corollary}

\bibliographystyle{amsplain}

\end{document}